\documentclass{amsart}
\usepackage{amssymb}
\usepackage[abbrev]{amsrefs}
\numberwithin{equation}{section}
\def\N{\mathbb N}
\def\R{\mathbb R}
\def\S{\mathbb S}
\providecommand{\norm}[1]{\left\lVert#1\right\rVert}
\newcommand{\ip}[2]{\left\langle #1, #2 \right\rangle}

\providecommand{\abs}[1]{\left\lvert#1\right\rvert}
\DeclareMathOperator*{\Argmin}{argmin}

\DeclareMathOperator{\Fix}{\mathcal{F}}
\DeclareMathOperator{\AC}{\mathcal{A}}

\DeclareMathOperator{\Diam}{diam}
\newcommand{\sk    }[1]{\left( #1 \right)}
\newcommand{\ck    }[1]{\left\{#1 \right\}}
\newcommand{\CAT}{\textup{CAT}}

\theoremstyle{plain}
\newtheorem{theorem}{Theorem}[section]
\newtheorem{lemma}[theorem]{Lemma}

\newtheorem{corollary}[theorem]{Corollary}
\newtheorem{example}[theorem]{Example}
\theoremstyle{definition}

\theoremstyle{remark}
\newtheorem{remark}[theorem]{Remark}
\allowdisplaybreaks
\title[Fixed points of vicinal mappings in geodesic spaces]
{Existence and approximation of fixed points of vicinal mappings in geodesic spaces} 
\author[F. Kohsaka]{Fumiaki Kohsaka}
\address[F. Kohsaka]
{Department of Mathematical Sciences, Tokai University, 
Kitakaname, Hiratsuka, Kanagawa 259-1292, Japan}
\email{f-kohsaka@tsc.u-tokai.ac.jp}
\begin{document}
\begin{abstract}
We propose the concepts of vicinal mappings 
and firmly vicinal mappings in metric spaces. 
We obtain fixed point and convergence theorems 
for these mappings in complete geodesic spaces with 
curvature bounded above by one 
and apply our results to convex optimization in such spaces. 
\end{abstract}
\maketitle
\section{Introduction}\label{sec:intro}

In this paper, we first introduce the 
classes of vicinal mappings and firmly vicinal mappings 
in metric spaces. 
We next obtain fixed point and convergence theorems 
for such mappings 
in complete $\CAT(1)$ spaces such that the distance of 
two arbitrary points in the space is less than $\pi/2$.  
Since the resolvents of convex functions 
proposed by Kimura and Kohsaka~\cite{MR3463526} 
are firmly vicinal, 
we can apply our results to convex optimization 
in such spaces. 

The problem of finding fixed points of nonexpansive mappings 
is strongly related to convex optimization in 
Hadamard spaces, i.e., complete $\CAT(0)$ spaces.  
In fact, it is known~\cites{MR3241330, MR3346750, MR1360608, MR1651416} that 
if $X$ is an Hadamard space and $f$ 
is a proper lower semicontinuous convex function 
of $X$ into $(-\infty, \infty]$, then 
the resolvent $J_{f}$ of $f$, which is given by 
\begin{align}\label{eq:classical-resolvent}
 J_{f} x = \Argmin_{y\in X} \ck{f(y) + \frac{1}{2}d(y, x)^2}
\end{align}
for all $x\in X$, 
is a well-defined nonexpansive mapping of $X$ into itself 
such that the fixed point set $\Fix(J_{f})$ of $J_{f}$
coincides with the set $\Argmin_{X} f$ of all minimizers of $f$. 
It is also known~\cite{MR3206460}*{Proposition~3.3} that 
$J_{f}$ is firmly nonexpansive, i.e., 
\begin{align*}
 d\left(J_fx, J_{f}y\right) 
 \leq d\bigl(\alpha x\oplus (1-\alpha) J_{f}x, 
\alpha y\oplus (1-\alpha) J_{f}y\bigr)
\end{align*}
whenever $x,y\in X$ and $\alpha \in (0,1)$. 
Thus we can apply the fixed point theory 
for nonexpansive mappings to the problem of 
minimizing convex functions in the space. 
See also~\cites{MR2798533, MR2548424} and~\cite{MR0390843} 
on the resolvents of convex functions in Hilbert and Banach spaces, respectively. 

In 2013, using the resolvent $J_{f}$ 
given by~\eqref{eq:classical-resolvent},  
Ba{\v{c}}{\'a}k~\cite{MR3047087}*{Theorem~1.4} obtained 
a $\Delta$-convergence theorem on the 
proximal point algorithm for 
convex functions in Hadamard spaces, 
which generalizes the corresponding result by 
Br{\'e}zis and Lions~\cite{MR491922}*{Th\'eor\`eme~9} 
in Hilbert spaces to more general Hadamard spaces.  
This algorithm was first introduced by Martinet~\cite{MR0298899} 
for variational inequality problems 
and generally studied by Rockafellar~\cite{MR0410483} 
for maximal monotone operators in Hilbert spaces. 
See also Bruck and Reich~\cite{MR470761} 
on some related results for strongly nonexpansive mappings in Banach spaces. 
Recently, Kimura and Kohsaka~\cite{MR3574140} 
obtained existence and convergence theorems 
on two modified proximal point algorithms 
for convex functions in Hadamard spaces. 

On the other hand, 
Ohta and P{\'a}lfia~\cite{MR3396425}*{Definition~4.1 and Lemma~4.2}
showed that the resolvent $J_f$ in~\eqref{eq:classical-resolvent} 
is well defined also in a complete $\CAT(1)$ space 
such that $\Diam(X)<\pi/2$, where $\Diam(X)$ 
denotes the diameter of $X$. 
Using this result, they~\cite{MR3396425}*{Theorem~5.1} 
studied the proximal point algorithm for convex functions in such spaces. 
We note that if $X$ is a complete $\CAT(1)$ space 
such that $\Diam(X)<\pi/2$, then 
every sequence in $X$ has 
a $\Delta$-convergent subsequence 
and every proper lower semicontinuous convex function 
of $X$ into $(-\infty, \infty]$ has a minimizer; see~\cite{MR2508878}*{Corollary~4.4} 
and~\cite{MR3463526}*{Corollary~3.3}, respectively. 
Thus the condition $\Diam(X)<\pi/2$ 
for a complete $\CAT(1)$ space $X$ 
can be seen as a counterpart of 
the boundedness condition for an Hadamard space $X$.   

Considering the geometric difference between 
Hadamard spaces and complete $\CAT(1)$ spaces, 
Kimura and Kohsaka~\cite{MR3463526}*{Definition~4.3} 
recently  
introduced the concept of resolvents of convex functions 
in complete $\CAT(1)$ spaces as follows. 
Let $X$ be a complete $\CAT(1)$ space 
which is admissible, i.e., 
\begin{align}\label{eq:intro-admissible}
 d(v,v') < \frac{\pi}{2}
\end{align}
for all $v,v'\in X$ 
and $f$ a proper lower semicontinuous 
convex function of $X$ into $(-\infty, \infty]$. 
It is known~\cite{MR3463526}*{Theorem~4.2} that 
the resolvent $R_{f}$ of $f$, which is given by 
\begin{align}\label{eq:R_f-intro}
 R_{f} x = \Argmin_{y\in X} \bigl\{f(y)+\tan d(y, x)\sin d(y, x)\bigr\}
\end{align}
for all $x\in X$, is a well-defined mapping of $X$ into itself.  
It is also known~\cite{MR3463526}*{Theorem~4.6} that 
$\Fix(R_f)$ coincides with $\Argmin_X f$, the inequality 
\begin{align}
 \begin{split}\label{eq:firmly-vicinal-intro}
  &\bigl(C_x^2(1+C_y^2)C_y+C_y^2(1+C_x^2)C_x\bigr)
  \cos d(R_fx,R_fy) \\
  &\quad \geq C_x^2 (1+C_y^2) \cos d(R_fx,y) 
 +C_y^2 (1+C_x^2) \cos d(R_fy,x) 
 \end{split}
\end{align}
holds for all $x,y\in X$, and $R_{f}$ is firmly spherically nonspreading, i.e., 
 \begin{align}\label{eq:f-s-nonsp-intro}
 (C_x + C_y) \cos^2 d(R_fx,R_fy) 
  \geq 2\cos d(R_fx,y)\cos d(R_fy,x) 
\end{align}
for all $x,y\in X$, 
where $C_z=\cos d(R_fz,z)$ for all $z\in X$. 
The inequality~\eqref{eq:firmly-vicinal-intro} means that 
$R_f$ is firmly vicinal in the sense of this paper. 

Moreover, Kimura and Kohsaka~\cite{MR3463526} obtained 
fixed point theorems and $\Delta$-convergence theorems for 
firmly spherically nonspreading mappings 
and applied them to convex optimization in 
complete $\CAT(1)$ spaces. 
However, they did not study the fixed point problem for 
mappings satisfying~\eqref{eq:firmly-vicinal-intro}. 
Since~\eqref{eq:firmly-vicinal-intro} is 
stronger than~\eqref{eq:f-s-nonsp-intro}, 
we can obtain fixed point and $\Delta$-convergence theorems 
which are independent of the results in~\cite{MR3463526}.  

More recently, applying the resolvent $R_{f}$  
given by~\eqref{eq:R_f-intro}, 
Kimura and Kohsaka~\cite{KimuraKohsaka-PPA-CAT1} 
and 
Esp{\'{\i}}nola and Nicolae~\cite{MR3597363} 
independently studied the proximal point algorithm 
for convex functions in complete $\CAT(\kappa)$ spaces 
with a positive real number $\kappa$. 

This paper is organized as follows. 
In Section~\ref{sec:pre}, we recall some definitions and 
results needed in this paper. 
In Section~\ref{sec:lem}, we give the definitions of 
vicinal mappings and firmly vicinal mappings 
in metric spaces such that 
the distance of two arbitrary points 
is less than or equal to $\pi/2$; 
see~\eqref{eq:vicinal} and~\eqref{eq:firmly-vicinal}. 
In Section~\ref{sec:main}, we obtain 
a fixed point theorem for vicinal mappings 
and a $\Delta$-convergence theorem 
for firmly vicinal mappings 
in admissible complete $\CAT(1)$ spaces; 
see Theorems~\ref{thm:vicinal-fpt} 
and~\ref{thm:firmly-vicinal-conv}, respectively.   
We also apply our results to convex optimization 
in such spaces; see Corollary~\ref{cor:cmp}.  
In Section~\ref{sec:positive-kappa}, 
we define the concepts of $\kappa$-vicinal mappings 
and firmly $\kappa$-vicinal mappings and obtain 
two corollaries of our results 
in complete $\CAT(\kappa)$ spaces 
with a positive real number $\kappa$; 
see Corollaries~\ref{cor:vicinal-fpt-kappa} 
and~\ref{cor:firmly-vicinal-conv-kappa}. 

\section{Preliminaries}
\label{sec:pre}

Throughout this paper, we denote by 
$\N$ and $\R$ the sets of 
all positive integers and all real numbers, 
respectively. 
We denote by $X$ a metric space with metric $d$. 
The diameter of $X$ is denoted by $\Diam(X)$. 
The closed ball with radius $r\geq 0$ 
centered at $p\in X$ is denoted by $S_r[p]$. 
For a mapping $T$ of $X$ into itself, 
we denote by $\Fix(T)$ the set of all $u\in X$ 
such that $Tu=u$. 
For a function $f$ of $X$ into $(-\infty, \infty]$,  
we denote by $\Argmin_X f$ or $\Argmin_{y\in X} f(y)$ 
the set of all $u\in X$ such that $f(u)=\inf f(X)$. 
In the case where $\Argmin_X f = \{p\}$ for some $p\in X$, 
we identify $\Argmin_X f$ with $p$. 

A mapping $T$ of $X$ into itself is said to be 
asymptotically regular if 
\begin{align*}
 \lim_{n\to \infty}d(T^{n+1}x,T^{n}x) = 0
\end{align*}
for all $x\in X$. For a sequence $\{x_n\}$ in $X$, 
the asymptotic center $\AC\bigl(\{x_n\}\bigr)$ 
of $\{x_n\}$ is defined by 
\begin{align*}
 \AC\bigl(\{x_n\}\bigr) 
 =\left\{z\in X: \limsup_{n\to \infty}d(x_n, z)
 =\inf_{y\in X}\limsup_{n\to \infty}d(x_n, y)\right\}. 
\end{align*}
The sequence $\{x_n\}$ is said to be $\Delta$-convergent to 
a point $p\in X$ if 
\begin{align*}
 \AC\bigl(\{x_{n_i}\}\bigr) = \{p\}
\end{align*}
for each subsequence $\{x_{n_i}\}$ of $\{x_n\}$. 
If $X$ is a Hilbert space, then 
the sequence $\{x_n\}$ is $\Delta$-convergent to $p$ 
if and only if it is weakly convergent to the point.  
For a sequence $\{x_n\}$ in $X$, 
we denote by $\omega_{\Delta}\bigl(\{x_n\}\bigr)$ 
the set of all $z\in X$ such that 
there exists a subsequence of $\{x_n\}$ 
which is $\Delta$-convergent to $z$. 
See~\cites{MR3241330, MR2508878, MR2416076} for more details 
on the concept of $\Delta$-convergence. 

Let $\lambda$ be a positive real number. 
A metric space $X$ is said to be $\lambda$-geodesic if 
for each $x,y\in X$ with $d(x,y)<\lambda$, 
there exists a mapping $c\colon [0,l]\to X$ such that 
$c(0)=x$, $c(l)=y$, 
and 
\begin{align*}
 d\bigl(c(t_1), c(t_2)\bigr) = \abs{t_1-t_2}
\end{align*}
for all $t_1,t_2\in [0,l]$, where $l=d(x,y)$. 
The mapping $c$ is called a geodesic 
from $x$ to $y$. 
In this case, the geodesic segment $[x,y]$ is defined by 
\begin{align*}
 [x,y] = \{c(t): 0\leq t\leq l\} 
\end{align*}
and the point $\alpha x\oplus (1-\alpha)y$ 
is defined by 
\begin{align*}
 \alpha x\oplus (1-\alpha)y = c\bigl((1-\alpha)l\bigr)
\end{align*}
for all $\alpha\in [0,1]$. 
A subset $C$ of a $\lambda$-geodesic space $X$ 
such that $d(v,v')<\lambda$ for all $v,v'\in C$ 
is said to be 
convex if 
\begin{align*}
 \alpha x\oplus (1-\alpha)y \in C
\end{align*}
whenever $x,y\in C$, 
$c$ is a geodesic from $x$ to $y$, 
and $\alpha\in [0,1]$. 
We note that the set $[x,y]$ and 
the point $\alpha x \oplus (1-\alpha)y$ 
depend on the choice of a geodesic $c$ 
from $x$ to $y$.  
However, they are determined uniquely 
if the space $X$ is uniquely $\lambda$-geodesic, i.e., 
there exists a unique geodesic 
from $x$ to $y$ for each $x,y\in X$ with $d(x,y)<\lambda$. 
See Ba{\v{c}}{\'a}k~\cite{MR3241330} and 
Bridson and Haefliger~\cite{MR1744486} 
for more details on geodesic spaces. 

Let $H$ be a real Hilbert space with inner product
$\ip{\,\cdot\,}{\,\cdot\,}$ and the induced norm $\norm{\,\cdot\,}$ 
and $S_H$ the unit sphere of $H$. 
The spherical metric $\rho_{S_H}$ on $S_H$ is defined by 
\begin{align*}
 \rho_{S_H}(x,y) = \arccos \ip{x}{y}
\end{align*}
for all $x,y\in S_H$. It is known that $\left(S_H, \rho_{S_H}\right)$ 
is a uniquely $\pi$-geodesic complete metric space whose 
metric topology coincides with the relative norm topology on $S_H$. 
If $x,y\in S_H$ and $0<\rho_{S_H}(x,y) < \pi$, 
then the unique geodesic $c$ from $x$ to $y$ is given by 
\begin{align*}
 c(t) = (\cos t) x + (\sin t)\cdot \frac{y-\ip{x}{y}x}{\norm{y-\ip{x}{y}x}} 
\end{align*}
for all $t\in [0,\rho_{S_H}(x,y)]$. 
The space $\left(S_H, \rho_{S_H}\right)$ is called a 
Hilbert sphere. 
We denote by $\S^2$ the unit sphere of the 
three dimensional Euclidean space $\R^3$ 
with the spherical metric $\rho_{\S^2}$ on $\S^2$.  

Let $X$ be a $\pi$-geodesic metric space 
and $x_1,x_2,x_3$ points in $X$ satisfying 
\begin{align}\label{eq:perimeter}
 d(x_1,x_2) + d(x_2,x_3) + d(x_3,x_1) <2\pi. 
\end{align}
According to~\cite{MR1744486}*{Lemma 2.14 in Chapter~I.2}, 
there exist $\bar{x}_1, \bar{x}_2, \bar{x}_3\in \S^2$ such that 
$d(x_{i}, x_{i+1})=\rho_{\S^2}(\bar{x}_{i}, \bar{x}_{i+1})$ 
for all $i\in \{1,2,3\}$, where $x_4=x_1$ and 
$\bar{x}_4=\bar{x}_1$. 
The sets $\Delta$ and $\bar{\Delta}$ given by 
\begin{align*}
 \Delta = [x_1,x_2] \cup [x_2,x_3] \cup [x_3,x_1] 
 \quad \textrm{and} \quad 
 \bar{\Delta} = [\bar{x}_1,\bar{x}_2] \cup [\bar{x}_2,\bar{x}_3] 
 \cup [\bar{x}_3,\bar{x}_1] 
\end{align*}
are called a geodesic triangle with vertices $x_1,x_2,x_3$ and 
a comparison triangle for $\Delta$ in $\S^2$, respectively. 
A point $\bar{p}\in \bar{\Delta}$ is called a 
comparison point for $p\in \Delta$ if 
\begin{align*}
 p\in [x_i, x_j], \quad 
 \bar{p} \in [\bar{x}_i, \bar{x}_j], \quad 
 \textrm{and} \quad 
 d(x_i, p) = \rho_{\S^2} (\bar{x}_i, \bar{p}) 
\end{align*}
for some distinct $i,j\in \{1,2,3\}$. 

A metric space $X$ is said to be a $\CAT(1)$ space 
if it is $\pi$-geodesic and 
\begin{align*}
 d(p, q) \leq \rho_{\S^2} (\bar{p}, \bar{q})
\end{align*}
whenever $\Delta$ is a geodesic triangle 
with vertices $x_1,x_2,x_3\in X$ 
satisfying~\eqref{eq:perimeter}, 
$\bar{\Delta}$ is a comparison triangle for $\Delta$ in $\S^2$, 
and $\bar{p}, \bar{q}\in \bar{\Delta}$ are comparison points for 
$p, q\in \Delta$, respectively. 
In this case, $X$ is uniquely $\pi$-geodesic. 
It is known that Hilbert spaces, Hilbert spheres, 
and Hadamard spaces are complete $\CAT(1)$ spaces. 
See Ba{\v{c}}{\'a}k~\cite{MR3241330}, 
Bridson and Haefliger~\cite{MR1744486}, 
and Goebel and Reich~\cite{MR744194} for more details 
on Hadamard spaces, $\CAT(\kappa)$ spaces with a real number 
$\kappa$, and Hilbert spheres, respectively. 

A $\CAT(1)$ space $X$ is said to be admissible 
if~\eqref{eq:intro-admissible} holds for all $v,v'\in X$. 
A sequence $\{x_n\}$ in $X$ 
is said to be spherically bounded if 
\begin{align*}
 \inf_{y\in X} \limsup_{n\to \infty} 
 d(x_n, y) < \frac{\pi}{2}. 
\end{align*}
In particular, if $\Diam(X)< \pi /2$, 
then the space $X$ is admissible and every sequence in $X$ is 
spherically bounded. 

We know the following fundamental lemmas.   
\begin{lemma}[\cite{MR2508878}*{Proposition~4.1 and Corollary~4.4}]
\label{lem:E-FL} 
 Let $X$ be a complete $\CAT(1)$ space and 
 $\{x_n\}$ a spherically bounded sequence in $X$. 
 Then $\AC\bigl(\{x_n\}\bigr)$ is a singleton and 
 $\{x_n\}$ has a $\Delta$-convergent subsequence. 
\end{lemma}

\begin{lemma}[\cite{MR3213144}*{Proposition~3.1}]
\label{lem:KSY} 
 Let $X$ be a complete $\CAT(1)$ space and 
 $\{x_n\}$ a spherically bounded sequence in $X$ 
 such that $\{d(x_n, z)\}$ is convergent 
 for each $z$ in $\omega_{\Delta}\bigl(\{x_n\}\bigr)$. 
 Then $\{x_n\}$ is $\Delta$-convergent 
 to an element of $X$.  
\end{lemma}

\begin{lemma}[See, for instance,~\cite{MR3463526}*{Lemma~2.3}]
\label{lem:CAT1-CMS}
 Let $X$ be a $\CAT(1)$ space and 
 $x_1, x_2, x_3$ points in $X$ such that~\eqref{eq:perimeter} holds. 
 If $d(x_1,x_3) \leq \pi/2$, $d(x_2,x_3)\leq \pi/2$, and $\alpha \in [0,1]$, then 
 \begin{align*}
  \cos d\bigl(\alpha x_1 \oplus (1-\alpha) x_2, x_3\bigr) 
  \geq \alpha \cos d(x_1, x_3) + (1-\alpha) \cos d(x_2,x_3). 
 \end{align*}
\end{lemma}

Let $X$ be an admissible $\CAT(1)$ space and 
$f$ a function of $X$ into $(-\infty, \infty]$.  
Then $f$ is said to be proper if 
$f(a)\in \R$ for some $a\in X$. 
It is also said to be convex if 
\begin{align*}
 f\bigl(\alpha x \oplus (1-\alpha) y\bigr) 
 \leq \alpha f(x) + (1-\alpha) f(y)
\end{align*}
whenever $x,y\in X$ and $\alpha \in (0,1)$. 
If $C$ is a nonempty closed convex subset of $X$, 
then the indicator function $i_C$ for $C$, 
which is defined by $i_C(x)=0$ if $x\in C$ and $\infty$ if $x\in X\setminus C$, 
is a proper lower semicontinuous convex function 
of $X$ into $(-\infty, \infty]$. 
A function $g$ of $X$ into $[-\infty, \infty)$ is 
said to be concave if $-g$ is convex. 
See~\cites{MR1113394, Yokota-JMSJ16} 
on some examples of convex functions in 
$\CAT(1)$ spaces. 

It is known~\cite{MR3463526}*{Theorem~4.2} that 
if $X$ is an admissible complete $\CAT(1)$ space, 
 $f$ is a proper lower semicontinuous 
 convex function of $X$ into $(-\infty, \infty]$, 
 and $x\in X$, 
 then there exists a unique $\hat{x}\in X$ such that 
 \begin{align*}
   f(\hat{x}) + \tan d(\hat{x}, x) \sin d(\hat{x}, x) 
   = \inf_{y\in X} \bigl\{f(y) + \tan d(y, x) \sin d(y, x)\bigr\}. 
 \end{align*}
Following~\cite{MR3463526}*{Definition~4.3}, 
we define the resolvent $R_{f}$ of $f$ by 
$R_fx = \hat{x}$ for all $x\in X$. 
In other words, $R_f$ is given by~\eqref{eq:R_f-intro} for all $x\in X$. 
The resolvent of the indicator function $i_C$ for 
a nonempty closed convex subset $C$ of $X$ 
coincides with the metric projection $P_C$ of $X$ onto $C$, i.e., 
\begin{align*}
\begin{split}
 R_{i_C}(x)
&=\Argmin_{y\in X} 
 \bigl\{i_C(y) + \tan d(y, x)\sin d(y, x)\bigr\} \\
&=\Argmin_{y\in C} \tan d(y, x)\sin d(y, x) 
 =\Argmin_{y\in C} d(y, x) = P_Cx   
\end{split}
\end{align*}
for all $x\in X$. 

We recently obtained the following maximization theorem. 

\begin{theorem}[\cite{KimuraKohsaka-PPA-CAT1}*{Theorem~4.1}]
\label{thm:argmax}
 Let $X$ be an admissible complete $\CAT(1)$ space, 
 $\{z_n\}$ a spherically bounded sequence in $X$, 
 $\{\beta_n\}$ a sequence of positive real numbers 
 such that $\sum_{n=1}^{\infty}\beta_n =\infty$, 
 and $g$ the real function on $X$ defined by 
 \begin{align}
  g(y) = \liminf_{n\to \infty} \frac{1}{\sum_{l=1}^{n}\beta_l} 
  \sum_{k=1}^{n} \beta_k \cos d(y, z_k)
 \end{align} 
 for all $y\in X$. 
 Then $g$ is a concave and nonexpansive function 
 of $X$ into $[0,1]$ and $g$ has a unique maximizer. 
\end{theorem}

We say that a real function $f$ on a nonempty subset $I$ of $\R$ is 
nondecreasing if 
$f(s_1)\leq f(s_2)$ whenever $s_1, s_2\in I$ and $s_1\leq s_2$. 
We also say that $f$ is nonincreasing if $-f$ is nondecreasing. 
It is clear that if 
$A$ is a nonempty bounded subset of $\R$, 
$I$ is a closed subset of $\R$ which contains $A$, 
and $f$ is a continuous and nondecreasing real function on $I$, 
then $f(\sup A) =\sup f(A)$ and $f(\inf A) =\inf f(A)$. 
In fact, setting $\alpha =\sup A$, we have 
$\alpha \in I$ by the closedness of $I$. 
Since $s\leq \alpha$ for all $s\in A$ 
and $f$ is nondecreasing, we have 
$f(s)\leq f(\alpha)$ for all $s\in A$. 
Thus we obtain $\sup f(A)\leq f(\alpha)$. 
On the other hand, the definition of $\alpha$ 
implies that there exists a 
sequence $\{s_n\}$ in $A$ converging to $\alpha$. 
Since $f$ is continuous, we have 
$f(\alpha) = \lim_{n} f(s_n) \leq \sup f(A)$. 
Thus we have $f(\alpha)=\sup f(A)$. 
The second equality can be shown similarly. 
Using these two equalities, we can show the following.   

\begin{lemma}\label{lem:limsup-liminf}
 Let $I$ be a nonempty closed subset of $\R$, 
 $\{t_n\}$ a bounded sequence in $I$, 
 and $f$ a continuous real function on $I$. 
 Then the following hold.   
 \begin{enumerate}
  \item[(i)] If $f$ is nondecreasing, 
    then $f(\limsup_n t_n)= \limsup_n f(t_n)$; 
  \item[(ii)] if $f$ is nonincreasing, 
    then $f(\limsup_n t_n)= \liminf_n f(t_n)$.  
 \end{enumerate}
\end{lemma}

\section{Vicinal mappings and firmly vicinal mappings}
\label{sec:lem}

In this section, 
motivated by the fact that the resolvent $R_{f}$ 
defined by~\eqref{eq:R_f-intro} satisfies the
inequality~\eqref{eq:firmly-vicinal-intro}, 
we give the definition of vicinal mappings 
and firmly vicinal mappings. 
We also study some fundamental properties of these mappings. 

Let $X$ be a metric space 
such that $d(v,v')\leq \pi/2$ for all $v,v'\in X$, 
$T$ a mapping of $X$ into itself, 
and $C_z$ the real number given by 
\begin{align}\label{eq:C_z}
 C_z=\cos d(Tz,z)
\end{align}
for all $z\in X$. 

The mapping $T$ is said to be 
\begin{itemize}
 \item vicinal if 
\begin{align}
 \begin{split}\label{eq:vicinal}
  & \bigl(C_x^2 (1+C_y^2)+C_y^2(1+C_x^2)\bigr)
  \cos d(Tx,Ty) \\
  &\quad \geq C_x^2 (1+C_y^2) \cos d(Tx,y) 
 +C_y^2 (1+C_x^2) \cos d(Ty,x) 
 \end{split}
\end{align}
for all $x,y\in X$; 
 \item firmly vicinal if 
\begin{align}
 \begin{split}\label{eq:firmly-vicinal}
  & \bigl(C_x^2(1+C_y^2)C_y+C_y^2(1+C_x^2)C_x\bigr)
  \cos d(Tx,Ty) \\
  &\quad \geq C_x^2 (1+C_y^2) \cos d(Tx,y) 
 +C_y^2(1+C_x^2) \cos d(Ty,x) 
 \end{split}
\end{align}
for all $x,y\in X$.  
\end{itemize}
Recall that $T$ is said to be 
\begin{itemize}
 \item spherically nonspreading~\cite{MR3463526} if 
\begin{align*}
  \cos^2 d(Tx,Ty) 
  \geq \cos d(Tx,y)\cos d(Ty,x) 
\end{align*}
for all $x,y\in X$; 
 \item firmly spherically nonspreading~\cite{MR3463526} if 
\begin{align*}
 (C_x + C_y) \cos^2 d(Tx,Ty) 
  \geq 2\cos d(Tx,y)\cos d(Ty,x) 
\end{align*}
for all $x,y\in X$; 
 \item quasi-nonexpansive if $\Fix(T)$ is nonempty 
and $d(Tx,y)\leq d(x,y)$ 
for all $x\in X$ and $y\in \Fix(T)$. 
\end{itemize}

Since $C_z\leq 1$ for all $z\in X$, 
every firmly spherically nonspreading mapping 
is spherically nonspreading. We know the following result. 
\begin{lemma}[\cite{MR3463526}*{Theorem~4.6}]
\label{lem:resolvent}
 Let $X$ be an admissible complete $\CAT(1)$ space, 
 $f$ a proper lower semicontinuous convex function 
 of $X$ into $(-\infty, \infty]$, and $R_f$ the resolvent of $f$. 
 Then $R_f$ is a firmly vicinal mapping of $X$ into itself 
 such that $\Fix(R_f)$ coincides with $\Argmin_X f$. 
\end{lemma}

We first show the following fundamental lemma. 

\begin{lemma}\label{lem:vicinal-relation}
 Let $X$ be a metric space such that $d(v,v')\leq \pi/2$ 
 for all $v,v'\in X$ and 
 $T$ a mapping of $X$ into itself. 
 Then the following hold.  
 \begin{enumerate}
  \item[(i)] Suppose that $T$ is firmly vicinal. 
 Then $T$ is vicinal. 
 Further, if $d(v,v')<\pi/2$ for all $v,v'\in X$, 
 then $T$ is firmly spherically nonspreading; 
  \item[(ii)] if $T$ is firmly vicinal and $\Fix(T)$ is nonempty, 
 then 
 \begin{align*}
   \cos d(Tx,x) \cos d(Tx,y) \geq \cos d(x,y)
 \end{align*}
 for all $x\in X$ and $y\in \Fix(T)$; 
  \item[(iii)] if $T$ is vicinal and $\Fix(T)$ is nonempty, 
 then it is quasi-nonexpansive; 
  \item[(iv)] if $d(v,v')<\pi/2$ for all $v,v'\in X$, 
 $T$ is firmly vicinal, and $\Fix(T)$ is nonempty, 
 then $T$ is asymptotically regular. 
  \end{enumerate}
\end{lemma}

\begin{proof}
 Let $C_z$ be the real number given by~\eqref{eq:C_z} 
 for all $z\in C$. 
 We first prove~(i). 
 Suppose that $T$ is firmly vicinal. 
 It is obvious that $T$ is vicinal 
 since $C_z \leq 1$ for all $z\in X$. 
 Further, suppose that $d(v,v')<\pi/2$ for all $v,v'\in X$. 
 Then, using an idea in~\cite{MR3463526}*{Theorem~4.6},  
 we show that $T$ is firmly spherically nonspreading. 
 Let $x,y\in X$ be given. 
 By the definition of firm vicinality and 
 the inequality of arithmetic and geometric means, 
 we have 
\begin{align}
 \begin{split}\label{eq:lem:vicinal-relation-a}
  & \bigl(C_x^2 (1+C_y^2)C_y+C_y^2 (1+C_x^2)C_x\bigr)
  \cos d(Tx,Ty) \\
  &\geq 
  C_x^2 \cos d(Tx,y) +C_y^2 \cos d(Ty,x) +
  C_x^2 C_y^2 \bigl(\cos d(Tx,y) + \cos d(Ty,x)\bigr) \\
  &\geq 2C_xC_y(1+C_xC_y) 
  \sqrt{\cos d(Tx,y) \cos d(Ty,x)}. 
 \end{split}
\end{align}
 On the other hand, we have 
 \begin{align}\label{eq:lem:vicinal-relation-b}
   C_x^2 (1+C_y^2)C_y+C_y^2 (1+C_x^2)C_x
   =C_xC_y (C_x+C_y)(1+C_xC_y). 
 \end{align}
 Noting that $C_xC_y>0$, 
 we have from~\eqref{eq:lem:vicinal-relation-a} 
 and~\eqref{eq:lem:vicinal-relation-b} that 
 \begin{align*}
  (C_x+C_y)^2 \cos^2 d(Tx, Ty) 
  \geq 4 \cos d(Tx,y) \cos d(Ty,x).  
 \end{align*}
 Since $2\geq C_x+C_y$, we know that 
 $T$ is firmly spherically nonspreading. 

 We next prove~(ii). 
 Suppose that $T$ is firmly vicinal and $\Fix(T)$ is nonempty. 
 Let $x\in X$ and $y\in \Fix(T)$ be given. 
 Since $Ty=y$ and $C_y=1$, we have 
 \begin{align*}
  \bigl(2C_x^2+(1+C_x^2)C_x\bigr) \cos d(Tx,y) 
  \geq 2C_x^2 \cos d(Tx,y) + (1+C_x^2) \cos d(y,x) 
 \end{align*}
 and hence 
 \begin{align*}
  (1+C_x^2)C_x \cos d(Tx,y) 
  \geq (1+C_x^2) \cos d(x,y). 
 \end{align*}
 Thus we obtain the conclusion. 
 
 We next prove~(iii). Suppose that $T$ is vicinal and $\Fix (T)$ 
 is nonempty. Let $x\in X$ and $y\in \Fix(T)$ be given.  
 Then we have 
 \begin{align*}
  \begin{split}
  &\bigl(2C_x^2+(1+C_x^2)\bigr) \cos d(Tx,y)
 \geq 2C_x^2 \cos d(Tx,y) 
 +(1+C_x^2) \cos d(y,x) 
  \end{split}
 \end{align*}
 and hence 
 \begin{align*}
  (1+C_x^2) \cos d(Tx,y) 
  \geq 
  (1+C_x^2) \cos d(x,y). 
 \end{align*}
 This implies that 
 \begin{align*}
  \cos d(Tx,y)\geq \cos d(x,y) 
 \end{align*}
 and hence we obtain the conclusion. 

 We finally prove~(iv). Suppose that $d(v,v')<\pi/2$ 
 for all $v,v'\in X$, $T$ is firmly vicinal, 
 and $\Fix(T)$ is nonempty. 
 Let $x\in X$ and $y\in \Fix(T)$ be given.  
 Then it follows from~(i) and~(iii) that $T$ is quasi-nonexpansive. 
 This implies that 
 \begin{align*}
  d(T^{n+1}x, y) \leq d(T^{n}x,y) \leq d(x, y) <\frac \pi 2
 \end{align*} 
 for all $n\in \N$
 and hence $\{d(T^{n}x,y)\}$ converges to some $l\in [0,\pi/2)$. 
 Then it follows from~(ii) that 
 \begin{align*}
   1  \geq \cos d(T^{n+1}x, T^{n}x) 
       \geq \frac{\cos d(T^{n}x, y)}{\cos d(T^{n+1}x, y)} 
       \to \frac{\cos l}{\cos l} = 1. 
 \end{align*}
 This yields that $\cos d(T^{n+1}x, T^{n}x)\to 1$. 
 Therefore we obtain the conclusion. 
\end{proof}

The following example shows that there exists a discontinuous 
spherically nonspreading mapping 
in an admissible complete $\CAT(1)$ space.  

\begin{example}\label{expl:sp}
 Let $\left(S_H, \rho_{S_H}\right)$ be a Hilbert sphere, 
 both $r$ and $\delta$ real numbers such that 
 \begin{align*}
  \frac{\pi}{8} < r < \frac{\pi}{4}, 
  \quad 0<\delta<1, 
  \quad \textrm{and} \quad 
  \cos \frac{\pi}{8} \leq \cos ^2 \frac{\delta \pi}{8}, 
 \end{align*}
 $p$ an element of $S_H$, $A=\{p\}$, 
 $B=S_{\delta \pi/8}[p]$, 
 $C=S_{\pi/8}[p]$, $X=S_r[p]$, 
 and both $P_A$ and $P_B$ the metric projections 
 of $X$ onto $A$ and $B$, respectively.  
 Then the mapping $T$ given by 
 \begin{align*}
  Tx=
  \begin{cases}
   P_A x & (x\in C); \\
   P_B x & (x\in X\setminus C)
  \end{cases}
 \end{align*}
 is a spherically nonspreading mapping of $X$ into itself.  
\end{example}

\begin{proof}
 We denote by $d$ the restriction of $\rho_{S_H}$ 
 on $X\times X$. 
 Since 
 \begin{align*}
  d(x,y)\leq d(x,p)+d(p,y) \leq 2r< \frac{\pi}{2}
 \end{align*}
 for all $x,y\in X$, the space $X$ is admissible. 
 We can see that $X$ is a convex subset of $S_H$. 
 In fact, if $x,y\in X$ and $\alpha \in [0,1]$, 
 then Lemma~\ref{lem:CAT1-CMS} implies that 
 \begin{align*}
  \begin{split}
   \cos d\bigl(\alpha x \oplus (1-\alpha)y, p\bigr) 
   &\geq \alpha \cos d(x,p) + (1-\alpha) \cos d(y, p) \\
   &\geq \alpha \cos r + (1-\alpha) \cos r 
  = \cos r.  
  \end{split}
 \end{align*}
 This implies that $d(\alpha x \oplus (1-\alpha)y, p)\leq r$ 
 and hence $\alpha x\oplus (1-\alpha)y \in X$. 
 Since $X$ is a nonempty closed convex subset of 
 the complete $\CAT(1)$ space $S_H$, 
 the space $X$ is also a complete $\CAT(1)$ space. 
 We can also see that $B$ is a convex subset of $X$.  

 By Lemmas~\ref{lem:resolvent} and~\ref{lem:vicinal-relation}, 
 we know that $P_A$ and $P_B$ are spherically nonspreading 
 and hence 
 \begin{align*}
  \cos ^2 d(Tx,Ty) \geq \cos d(Tx,y)\cos d(Ty, x)
 \end{align*}
 whenever $(x,y)\in C^2$ or 
 $(x,y)\in \left(X\setminus C\right)^2$. 
 Suppose that $x\in X\setminus C$ and $y\in C$. 
 Then we have $d(Tx,Ty) = d(P_Bx,p) \leq \delta \pi/8$ 
 and hence 
 \begin{align}\label{expl:sp-a}
  \cos ^2 d(Tx,Ty)\geq \cos ^2 \frac{\delta \pi}{8}. 
 \end{align}
 On the other hand, we have 
 $d(Ty, x)=d(p, x)>\pi/8$ and hence 
 \begin{align}\label{expl:sp-b}
  \cos d(Ty, x) < \cos \frac{\pi}{8}. 
 \end{align}
 By~\eqref{expl:sp-a} and~\eqref{expl:sp-b}, we have 
 \begin{align*}
   \cos d(Tx,y)\cos d(Ty, x) 
  \leq \cos d(Ty, x) 
  < \cos \frac{\pi}{8} 
  \leq \cos ^2 \frac{\delta\pi}{8} 
  \leq \cos ^2 d(Tx,Ty). 
 \end{align*}
 Therefore $T$ is spherically nonspreading. 
\end{proof}

\section{Existence and approximation of fixed points}
\label{sec:main}

In this section, we study the existence and approximation 
of fixed points of vicinal mappings 
and firmly vicinal mappings, respectively.   

Using Theorem~\ref{thm:argmax}, 
we obtain the following fixed point theorem for vicinal mappings 
in admissible complete $\CAT(1)$ spaces.  

\begin{theorem}\label{thm:vicinal-fpt}
 Let $X$ be an admissible complete $\CAT(1)$ space 
 and $T$ a vicinal mapping of $X$ into itself. 
 Then $\Fix(T)$ is nonempty if and only if 
 there exists $x\in X$ such that 
 $\{T^nx\}$ is spherically bounded and 
 $\sup_{n} d(T^{n}x, T^{n-1}x)< \pi/2$. 
\end{theorem}

\begin{proof}
 Let $C_z$ be the real number given by~\eqref{eq:C_z} 
 for all $z\in C$. 
 The only if part is obvious. In fact, if $\Fix(T)$ is nonempty 
 and $x\in \Fix(T)$, then we have 
 \begin{align*}
  \sup_{n} d(T^{n}x, T^{n-1}x) =0<\frac{\pi}{2} 
 \end{align*}
 and
 \begin{align*}
  \inf_{y\in X}\limsup_{n\to \infty} d(T^nx, y)
  =\inf_{y\in X}d(x, y) < \frac \pi 2,  
 \end{align*}
where the last inequality 
 follows from the admissibility of $X$. 

 We next prove the if part. Suppose that 
 there exists $x\in X$ such that 
 $\{T^nx\}$ is spherically bounded and 
 $\sup_{n} d(T^{n}x, T^{n-1}x)< \pi / 2$. 
 Set 
 \begin{align*}
  x_n=T^{n-1}x, \quad 
  \beta_n=\frac{C_{x_n}^2}{1+C_{x_n}^2}, 
  \quad 
  \textrm{and} 
  \quad 
  \sigma_n = \sum_{k=1}^{n}\beta_k
 \end{align*}
 for all $n\in \N$ and let $g$ be the real function on $X$ defined by 
 \begin{align*}
  g(y)=\liminf_{n\to \infty} \frac{1}{\sigma_n} 
  \sum_{k=1}^{n}\beta_k \cos d(y, x_{k+1})
 \end{align*}
 for all $y\in X$. 
 Since 
\begin{align*}
 \inf_{n}C_{x_n} = \inf_{n} \cos d(T^{n}x, T^{n-1}x)
=\cos \sk{\sup_{n} d(T^{n}x, T^{n-1}x)} > \cos \frac \pi 2 = 0, 
\end{align*} 
 we have 
 \begin{align*}
  \sigma_n= \sum_{k=1}^{n}\beta_k 
\geq \frac{1}{2}\sum_{k=1}^{n} C_{x_k}^2
\geq \sk{\inf_{m}C_{x_m}}^2\frac{n}{2} \to \infty 
 \end{align*}
as $n\to \infty$. 
Hence we have $\sum_{n=1}^{\infty}\beta_n = \infty$. 
Thus Theorem~\ref{thm:argmax} ensures that 
$g$ has a unique maximizer $p\in X$.   

On the other hand, the vicinality of $T$ implies that 
\begin{align*}
 \begin{split}
  &\bigl(C_{x_k}^2 (1+C_p^2)+C_p^2 (1+C_{x_k}^2)\bigr)
  \cos d(x_{k+1},Tp) \\
  &\quad \geq C_{x_k}^2 (1+C_p^2) \cos d(x_{k+1},p) 
 +C_p^2 (1+C_{x_k}^2) \cos d(Tp,x_k)    
 \end{split}
\end{align*}
and hence 
 \begin{align*}
  \begin{split}
  & \frac{C_{x_k}^2}{1+C_{x_k}^2} \cos d(x_{k+1},Tp) \\
  &\quad \geq \frac{C_{x_k}^2}{1+C_{x_k}^2} \cos d(x_{k+1},p) 
 +\frac{C_{p}^2}{1+C_{p}^2} 
 \bigl(\cos d(Tp, x_k) -\cos d(Tp, x_{k+1})\bigr)
  \end{split}
 \end{align*} 
 for all $k\in \N$. 
 This inequality yields 
 \begin{align*}
  \begin{split}
  &\frac{1}{\sigma_n}\sum_{k=1}^{n}\beta_k \cos d(x_{k+1},Tp) \\
  &\quad \geq \frac{1}{\sigma_n}\sum_{k=1}^{n}\beta_k \cos d(x_{k+1},p) 
 +\frac{1}{\sigma_n}\cdot \frac{C_{p}^2}{1+C_{p}^2} 
 \bigl(\cos d(Tp, x_1) -\cos d(Tp, x_{n+1})\bigr). 
  \end{split}  
 \end{align*}
 Taking the lower limit in this inequality, we obtain $g(Tp)\geq g(p)$. 
 Since $p$ is the unique maximizer of $g$, we conclude that 
 $Tp=p$. Therefore $T$ has a fixed point. 
\end{proof}

\begin{remark}
 In~\cite{MR3463526}*{Theorem~5.2}, 
 it was shown that if $X$ is an admissible complete $\CAT(1)$ space 
 and $T$ is a spherically nonspreading mapping of $X$ into itself, 
 then  $\Fix(T)$ is nonempty if and only if 
 there exists $x\in X$ such that 
 \begin{align*}
 \limsup_{n\to \infty} d(T^nx, Ty) < \frac{\pi}{2} 
 \end{align*}
 for all $y\in X$. 
 Note that Theorem~\ref{thm:vicinal-fpt} 
 is independent of this result. 
\end{remark}

As a direct consequence of Theorem~\ref{thm:vicinal-fpt}, 
we obtain the following corollary. 

\begin{corollary}\label{cor:vicinal-fpt}
 Let $X$ be a complete $\CAT(1)$ space such that 
 $\Diam(X)<\pi/2$. 
 Then every vicinal mapping $T$ of $X$ into itself has a fixed point.  
 \end{corollary}

Before obtaining a $\Delta$-convergence theorem, 
we show the following demiclosedness principle 
for vicinal mappings.  

\begin{lemma}\label{lem:demiclosed}
 Let $X$ be a metric space such that 
 $d(v,v')<\pi/2$ for all $v,v'\in X$,  
 $T$ a vicinal mapping of $X$ into itself, 
 $p$ an element of $X$, 
 and $\{x_n\}$ a sequence in $X$ such that 
 $\AC\bigl(\{x_n\}\bigr)=\{p\}$ and 
 $d(Tx_n, x_n)\to 0$. 
 Then $p$ is a fixed point of $T$. 
\end{lemma}

\begin{proof}
 Let $C_z$ be the real number given by~\eqref{eq:C_z} 
 for all $z\in C$. 
 Since $d(Tx_n, x_n)\to 0$, 
 we know that 
 \begin{align}\label{eq:lem:demiclosed-a}
  \lim_{n\to \infty} C_{x_n} = 1. 
 \end{align}
 On the other hand, since $t\mapsto \cos t$ is nonexpansive 
 and $d(Tx_n, x_n) \to 0$, we have 
 \begin{align*}
  \begin{split}
  \abs{\cos d(x_n, Tp) - \cos d(Tx_n, Tp)} 
   &\leq \abs{d(x_n, Tp) - d(Tx_n, Tp)}  \\
   &\leq d(x_n, Tx_n) \to 0
  \end{split}   
 \end{align*}
 and hence 
 \begin{align}\label{eq:lem:demiclosed-b}
  \lim_{n\to \infty}
 \bigl(d(x_n, Tp) - \cos d(Tx_n, Tp)\bigr) =0. 
 \end{align}
 The vicinality of $T$ implies that 
 \begin{align*}
  \begin{split}
   &\bigl(C_{x_n}^2(1+C_p^2) + C_p^2(1+C_{x_n}^2)\bigr)
   \cos d(Tx_n, Tp) \\
   & \geq C_{x_n}^2(1+C_p^2) \cos d(Tx_n, p)
 + C_p^2(1+C_{x_n}^2) \cos d(Tp, x_n)
  \end{split}
 \end{align*}
 and hence 
 \begin{align}
  \begin{split}\label{eq:lem:demiclosed-c} 
   & \cos d(Tx_n, Tp) \\
   & \geq \cos d(Tx_n, p)
 + \frac{C_p^2}{1+C_p^2}\cdot \frac{1+C_{x_n}^2}{C_{x_n}^2} 
   \bigl(\cos d(x_n, Tp) - \cos d(Tx_n, Tp)\bigr)
  \end{split}
 \end{align}
 for all $n\in \N$. 

 Using~\eqref{eq:lem:demiclosed-a},~\eqref{eq:lem:demiclosed-b}, 
 and~\eqref{eq:lem:demiclosed-c}, we have 
 \begin{align*}
   \liminf_{n\to \infty} \cos d(Tx_n, Tp) 
   \geq \liminf_{n\to \infty} \cos d(Tx_n, p). 
 \end{align*}
 Then it follows from Lemma~\ref{lem:limsup-liminf} that 
 \begin{align*}
  \begin{split}
   \cos \left(\limsup_{n\to \infty} d(Tx_n, Tp) \right) 
   &=\liminf_{n\to \infty} \cos d(Tx_n, Tp) \\
   &\geq \liminf_{n\to \infty} \cos d(Tx_n, p) 
   = \cos \left(\limsup_{n\to \infty} d(Tx_n, p)\right)    
  \end{split}
 \end{align*}
 and hence 
 \begin{align*}
  \limsup_{n\to \infty} d(Tx_n, Tp) 
  \leq \limsup_{n\to \infty} d(Tx_n, p). 
 \end{align*}
 It then follows from $d(Tx_n, x_n)\to 0$ that 
 \begin{align*}
 \begin{split}
  \limsup_{n\to \infty} d(x_n, Tp) 
  &= \limsup_{n\to \infty} d(Tx_n, Tp) \\
  &\leq \limsup_{n\to \infty} d(Tx_n, p) = \limsup_{n\to \infty} d(x_n, p).   
 \end{split}
 \end{align*}
 Thus, since $\AC\bigl(\{x_n\}\bigr)=\{p\}$, we conclude that $Tp=p$. 
\end{proof}

We next obtain the following $\Delta$-convergence theorem 
for firmly vicinal mappings in admissible complete $\CAT(1)$ spaces.  

\begin{theorem}\label{thm:firmly-vicinal-conv}
 Let $X$ be an admissible complete $\CAT(1)$ space 
 and $T$ a firmly vicinal mapping of $X$ into itself. 
 If $\Fix(T)$ is nonempty, then 
 $\{T^nx\}$ is $\Delta$-convergent to an element of 
 $\Fix (T)$ for each $x\in X$. 
\end{theorem}

\begin{proof}
 Let $x\in X$ be given. 
 By~(i) and~(iii) of Lemma~\ref{lem:vicinal-relation}, 
 $T$ is quasi-nonexpansive. 
 Combining this property with the admissibility of $X$, 
 we have 
 \begin{align*}
  d(T^{n+1}x, y) \leq d(T^{n}x, y) \leq d(x, y) < \frac \pi 2
 \end{align*}
 for each $y\in \Fix(T)$. 
 This gives us that the sequence $\{d(T^{n}x, y)\}$ 
 converges to an element of $[0,\pi/2)$ for each $y\in \Fix(T)$. 
 Since 
 \begin{align*}
  \inf_{y\in X} \limsup_{n\to \infty} d(T^nx, y) 
  \leq \inf_{y\in \Fix(T)} \limsup_{n\to \infty} d(T^nx, y) 
  \leq \inf_{y\in \Fix(T)} d(x, y) 
  < \frac{\pi}{2}, 
 \end{align*}
 the sequence $\{T^nx\}$ is spherically bounded. 
 
 Let $z$ be any element of $\omega_{\Delta}\bigl(\{T^nx\}\bigr)$. 
 By the definition of $\omega_{\Delta}\bigl(\{T^nx\}\bigr)$, 
 there exists a subsequence $\{T^{n_i}x\}$ of $\{T^nx\}$ 
 which is $\Delta$-convergent to $z$. 
 Then we have $\AC\bigl(\{T^{n_i}x\}\bigr)=\{z\}$. 
 By~(iv) of Lemma~\ref{lem:vicinal-relation}, 
 we know that $T$ is asymptotically regular and hence 
 \begin{align*}
  \lim_{i\to \infty} d\bigl(T(T^{n_i}x), T^{n_i}x\bigr) 
  = \lim_{n\to \infty} d(T^{n+1}x, T^{n}x) =0. 
 \end{align*} 
 According to Lemma~\ref{lem:demiclosed}, we have $z\in \Fix(T)$. 
 Thus $\omega_{\Delta}\bigl(\{T^nx\}\bigr)$ is contained by $\Fix(T)$. 
 Consequently, the real sequence $\{d(T^nx, z)\}$ 
 is convergent for each $z$ in $\omega_{\Delta}\bigl(\{T^nx\}\bigr)$. 
 Then, it follows from Lemma~\ref{lem:KSY} that 
 $\{T^nx\}$ is $\Delta$-convergent to some $p \in X$. 
 Since every subsequence of $\{T^nx\}$ is 
 also $\Delta$-convergent to $p$, we obtain 
 \begin{align*}
  \{p\} = \omega_{\Delta}\bigl(\{T^nx\}\bigr) 
  \subset \Fix(T). 
 \end{align*}  
 Therefore we conclude that $p$ is a fixed point of $T$.  
\end{proof}

\begin{remark}
 In~\cite{MR3463526}*{Theorem~6.5}, 
 it was shown that if $X$ is an admissible complete $\CAT(1)$ space 
 and $T$ is a firmly spherically nonspreading mapping of 
 $X$ into itself such that $\Fix(T)$ is nonempty and 
 \begin{align}\label{eq:fsn-cond}
 \limsup_{n\to \infty} d(y_n, Ty) < \frac{\pi}{2} 
 \end{align}
 whenever $\{y_n\}$ is a sequence in $X$ 
 which is $\Delta$-convergent to $y\in X$, 
 then $\{T^nx\}$ is $\Delta$-convergent to an element of 
 $\Fix (T)$ for each $x\in X$. 
 By Theorem~\ref{thm:firmly-vicinal-conv}, 
 we know that the assumption~\eqref{eq:fsn-cond} 
 is not needed for the special case where $T$ is firmly vicinal. 
\end{remark}

As a direct consequence of Corollary~\ref{cor:vicinal-fpt} and 
Theorem~\ref{thm:firmly-vicinal-conv}, we obtain the following corollary.   

\begin{corollary}\label{cor:firmly-vicinal-conv}
 Let $X$ be a complete $\CAT(1)$ space 
 such that $\Diam(X)<\pi/2$ 
 and $T$ a firmly vicinal mapping of $X$ into itself. 
 Then $\{T^nx\}$ is $\Delta$-convergent to an element of 
 $\Fix (T)$ for each $x\in X$. 
\end{corollary}

As a direct consequence of 
Lemma~\ref{lem:resolvent}, 
Theorems~\ref{thm:vicinal-fpt}, 
and~\ref{thm:firmly-vicinal-conv}, we obtain the following corollary.   

\begin{corollary}
\label{cor:cmp}
 Let $X$ be an admissible complete $\CAT(1)$ space, 
 $f$ a proper lower semicontinuous convex function 
 of $X$ into $(-\infty, \infty]$, and $R_f$ the resolvent of $f$. 
 Then $\Argmin_X f$ is nonempty if and only if 
 there exists $x\in X$ such that 
 $\{R_f^nx\}$ is spherically bounded and 
  $\sup_{n} d(R_f^{n}x, R_f^{n-1}x)<\pi/2$. 
 In this case, $\{R_f^nx\}$ is $\Delta$-convergent to 
 an element of $\Argmin_X f$ for each $x\in X$.   
\end{corollary}

\section{Results in $\textup{CAT}(\kappa)$ spaces with a positive $\kappa$}
\label{sec:positive-kappa}

In this section, 
we define the concepts of 
$\kappa$-vicinal mappings and firmly $\kappa$-vicinal mappings 
and obtain two corollaries of our results for these mappings 
in complete $\CAT(\kappa)$ spaces with a positive real number $\kappa$.  

Let $\kappa$ be a positive real number, $X$ a metric space such that 
$d(v,v')\leq \pi/(2\sqrt{\kappa})$ for all $v,v'\in X$, 
$T$ a mapping of $X$ into itself, and $\tilde{C}_z$ 
the real number defined by 
\begin{align*}
 \tilde{C}_z=\cos \sqrt{\kappa} d(Tz,z)
\end{align*}
for all $z\in X$. 
The mapping $T$ is said to be 
\begin{itemize}
 \item $\kappa$-vicinal if 
\begin{align*}
 \begin{split}
  & \bigl(\tilde{C}_x^2 (1+\tilde{C}_y^2)+\tilde{C}_y^2 (1+\tilde{C}_x^2)\bigr)
  \cos \sqrt{\kappa}d(Tx,Ty) \\
  &\quad \geq \tilde{C}_x^2 (1+\tilde{C}_y^2) \cos \sqrt{\kappa}d(Tx,y) 
 +\tilde{C}_y^2 (1+\tilde{C}_x^2) \cos \sqrt{\kappa}d(Ty,x) 
 \end{split}
\end{align*}
for all $x,y\in X$; 
 \item firmly $\kappa$-vicinal if 
\begin{align*}
 \begin{split}
  & \bigl(\tilde{C}_x^2(1+\tilde{C}_y^2)\tilde{C}_y+\tilde{C}_y^2 (1+\tilde{C}_x^2)\tilde{C}_x\bigr)
  \cos \sqrt{\kappa}d(Tx,Ty) \\
  &\quad \geq \tilde{C}_x^2 (1+\tilde{C}_y^2) \cos \sqrt{\kappa}d(Tx,y) 
 +\tilde{C}_y^2 (1+\tilde{C}_x^2) \cos \sqrt{\kappa}d(Ty,x) 
 \end{split}
\end{align*}
for all $x,y\in X$. 
\end{itemize}
Note that $1$-vicinal mappings and firmly $1$-vicinal mappings 
are coincident with vicinal mappings 
and firmly vicinal mappings, respectively. 

Let $\kappa$ be a positive real number, 
$D_{\kappa}$ the real number given by $D_{\kappa}=\pi/\sqrt{\kappa}$, 
and $(M_{\kappa}, \rho_{\kappa})$ the uniquely 
$D_{\kappa}$-geodesic space given by 
\begin{align*}
 (M_{\kappa}, \rho_{\kappa}) 
 = \left(\S^2, \frac{1}{\sqrt{\kappa}}\rho_{\S^2}\right). 
\end{align*}
A metric space $X$ is said to be a $\CAT(\kappa)$ space 
if it is $D_{\kappa}$-geodesic and 
\begin{align*}
 d(p, q) \leq \rho_{\kappa} (\bar{p}, \bar{q})
\end{align*}
whenever $\Delta$ is a geodesic triangle 
with vertices $x_1,x_2,x_3\in X$ 
satisfying 
\begin{align*}
 d(x_1,x_2) + d(x_2,x_3) + d(x_3,x_1) < 2D_{\kappa}, 
\end{align*}
$\bar{\Delta}$ is a comparison triangle for $\Delta$ in $M_{\kappa}$, 
and $\bar{p}, \bar{q}\in \bar{\Delta}$ are comparison points for 
$p, q\in \Delta$, respectively. 
It is obvious that $(X,d)$ is a complete $\CAT(\kappa)$ space 
such that $d(v,v')<D_{\kappa}/2$ for all $v,v'\in X$ 
if and only if $(X, \sqrt{\kappa} d)$ is 
an admissible complete $\CAT(1)$ space. 

As direct consequences of 
Theorems~\ref{thm:vicinal-fpt} and~\ref{thm:firmly-vicinal-conv}, 
we obtain the following two corollaries, respectively.  

\begin{corollary}\label{cor:vicinal-fpt-kappa}
 Let $\kappa$ be a positive real number, 
 $X$ a complete $\CAT(\kappa)$ space 
 such that $d(v,v')<D_{\kappa}/2$ for all $v,v'\in X$, 
 and $T$ a $\kappa$-vicinal mapping of $X$ into itself. 
 Then $\Fix(T)$ is nonempty if and only if 
 there exists $x\in X$ such that 
 \begin{align*}
  \inf_{y\in X} \limsup_{n\to \infty} 
  d(T^nx,y)<\frac{D_{\kappa}}{2} 
  \quad \textrm{and} \quad 
  \sup_{n} d(T^{n}x, T^{n-1}x)
 < \frac{D_{\kappa}}{2}. 
 \end{align*}
\end{corollary}

\begin{corollary}\label{cor:firmly-vicinal-conv-kappa}
 Let $\kappa$ be a positive real number, 
 $X$ a complete $\CAT(\kappa)$ space 
 such that $d(v,v')<D_{\kappa}/2$ for all $v,v'\in X$, 
 and $T$ a firmly $\kappa$-vicinal mapping of $X$ into itself. 
 If $\Fix(T)$ is nonempty, then 
 $\{T^nx\}$ is $\Delta$-convergent to an element of 
 $\Fix (T)$ for each $x\in X$. 
\end{corollary}

\section*{Acknowledgment}
The author would like to express his sincere appreciation 
to the anonymous referee and Professor Simeon Reich 
for their helpful comments on the 
original version of this paper. 
This work was supported by 
JSPS KAKENHI Grant Numbers 25800094 and 17K05372.

\end{document}